\documentclass[11pt]{amsart}
\usepackage{amsmath}
\usepackage{amsthm}
\usepackage{amsfonts}
\usepackage{graphicx}

\newtheorem{theorem}{Theorem}[section]

\newtheorem{proposition}[theorem]{Proposition}

\theoremstyle{definition}

\theoremstyle{remark}

\numberwithin{equation}{section}
\newcommand{\cone}{\mbox{$\times \hspace*{-0.258cm} \times$}}

\begin{document}

\title[Some minimal submanifolds generalizing the Clifford torus ]
{Some minimal submanifolds generalizing the Clifford torus }
\author[J. CHOE and J. Hoppe ]{JAIGYOUNG CHOE and JENS HOPPE}
\thanks{J.C. supported in part by NRF 2011-0030044, SRC-GAIA}

\address{Korea Institute for Advanced Study, Seoul, 02455, Korea}
\email{choe@kias.re.kr}

\address{KTH, Stockholm, Sweden}
\email{hoppe@kth.se}

\begin{abstract}
The Clifford torus is a product surface in $\mathbb S^3$ and it is {\it helicoidal}. It will be shown that more minimal submanifolds of $\mathbb S^n$ have these properties.
\end{abstract}

\maketitle

The Clifford torus is the simplest minimal surface in $\mathbb S^3$ besides the great sphere. Similarly in higher dimension we have a generalized Clifford torus $\mathbb S^p\left(\sqrt{\frac{p}{p+q}}\right)\times\mathbb S^q\left(\sqrt{\frac{q}{p+q}}\right)$ which is minimal in $\mathbb S^{p+q+1}$.

In Euclidean space there is an easy theorem that $\Sigma_1\times\Sigma_2$ is minimal in $\mathbb R^{n_1+n_2}$ if $\Sigma_1$ and $\Sigma_2$ are minimal in $\mathbb R^{n_1}$ and $\mathbb R^{n_2}$, respectively. While one cannot expect the same theorem to hold literally in $\mathbb S^n$, we will prove an analogous theorem as follows:
\begin{quote}
\it If $\Sigma_1^m$ is minimal in $\mathbb S^p $ and $\Sigma_2^n$ is minimal in $\mathbb S^q$, then \\$\sqrt{\frac{m}{m+n}}\,\Sigma_1^m\times\sqrt{\frac{n}{m+n}}\,\Sigma_2^n$ is minimal in $\mathbb S^{p+q+1}$.
\end{quote}

There is another way of proving the minimality of the Clifford torus $\Sigma$ in $\mathbb S^{3}$. It is well known that $\Sigma$ is (doubly) foliated by great circles and $\Sigma$  divides $\mathbb S^3$ into two congruent domains $D_1,D_2$. For every great circle $\ell$ in $\Sigma$ consider the rotation $\rho_\ell$ of $\mathbb S^3$ about $\ell$ by $180^\circ$. One can show that
\begin{equation}\label{helicoidal}
\rho_\ell(\Sigma)=\Sigma,\,\,\,\,\rho_\ell(D_1)=D_2,\,\,\,\, \rho_\ell(D_2)=D_1,\,\,\,\,\rho_\ell(p)=p,\,\,\forall\, p\in\ell.
\end{equation}
More generally, if a hypersurface $\Sigma^{n-1}$ of a Riemannian manifold $M^n$ has an isometry $\rho$ (in place of $\rho_\ell$) satisfying \eqref{helicoidal} at every point $p$ of $\Sigma^{n-1}$, $\Sigma^{n-1}$ is said to be {\it helicoidal} in $M$. In Proposition \ref{clifford} we show that the generalized Clifford torus $\mathbb S^p\left(1/{\sqrt{2}}\right)\times\mathbb S^p\left(1/{\sqrt{2}}\right)$ is helicoidal in $\mathbb S^{2p+1}$. In theorem \ref{helicoidalthm} we will prove that every helicoidal hypersurface of $M$ is minimal.

Recently Tkachev \cite{T} and Hoppe-Linardopoulos-Turgut \cite{HLT} found algebraic minimal hypersurfaces $N_1$ in $\mathbb S^{n^2-1}$ and $N_2$ in $\mathbb S^{2n^2-n-1}$, respectively:\\
$ N_1=\{(x_{11},x_{12},\ldots,x_{nn})\in\mathbb R^{n^2}:(x_{ij})\,\, {\rm is\,\, an \,\,{\it n\times n} \,\,matrix\,\,with\,\, zero\,\, determinant.}\}$\\
$\cap\,\,\mathbb S^{n^2-1};$\\
$N_2=\{(x_{11},x_{12},\ldots,x_{2n2n})\in\mathbb R^{4n^2}:(x_{ij})$ is a $2n\times 2n$ skew-symmetric matrix with  zero determinant.$\} \cap\,\,\mathbb S^{4n^2-1}$ is similar to a minimal hypersurface in $\mathbb S^{2n^2-n-1}.$
In this paper we give a new proof of their minimality, showing that they are helicoidal.

\section{product manifolds}
Let $\Sigma^m$ be an $m$-dimensional submanifold of $\mathbb S^p$ and $\Sigma^n$ an $n$-dimensional submanifold of $\mathbb S^q$. Denote by $a\Sigma^m$ and $b\Sigma^n$ the homothetic expansions of $\Sigma^m\subset\mathbb R^{p+1}$ and $\Sigma^n\subset\mathbb R^{q+1}$ with factors of $a$ and $b$, respectively.

\begin{theorem}
If $\Sigma_1^{n_1}$ is minimal in $\mathbb S^p$ and $\Sigma_2^{n_2}$ is minimal in $\mathbb S^q$, then $\sqrt{\frac{n_1}{n_1+n_2}}\Sigma_1^{n_1}\times\sqrt{\frac{n_2}{n_1+n_2}}\Sigma_2^{n_2}$ is minimal in $\mathbb S^{p+q+1}$.
\end{theorem}
\begin{proof}
Let $\varphi_1,\ldots,\varphi_{n_1}$ be the local coordinates of $\Sigma_1^{n_1}$ such that ${\bf m_1}:(\varphi_1,\ldots,\varphi_{n_1})\in D_1\subset\mathbb R^{n_1}\rightarrow{\bf m_1}(\varphi_1,\ldots,\varphi_{n_1})\in\Sigma_1^{n_1}\subset\mathbb S^p\subset\mathbb R^{p+1}$ is a local immersion. Similarly, $\varphi_{n_1+1},\ldots,\varphi_{n_1+n_2}$ are local coordinates of $\Sigma_2^{n_2}$ with a local immersion ${\bf m_2}:D_2\subset\mathbb R^{n_2}\rightarrow\Sigma_2^{n_2}\subset\mathbb S^q\subset\mathbb R^{q+1}$. Clearly $\sqrt{\frac{n_1}{n_1+n_2}}\Sigma_1^{n_1}\times\sqrt{\frac{n_2}{n_1+n_2}}\Sigma_2^{n_2}\subset\mathbb R^{p+q+2}$ is a subamnifold of $\mathbb S^{p+q+1}$. Let's define a local immersion ${\bf \hat{m}}:D_1\times D_2\subset\mathbb R^{n_1+n_2}\rightarrow\mathbb S^{p+q+1}\subset\mathbb R^{p+q+2}$ by
$${\bf \hat{m}}(\varphi_1,\ldots,\varphi_{n_1+n_2})=\left(\begin{array}{c}\sqrt{\frac{n_1}{n_1+n_2}}{\bf m_1}(\varphi_1,\ldots,\varphi_{n_1})\\\sqrt{\frac{n_2}{n_1+n_2}}{\bf m_2}(\varphi_{n_1+1},\ldots,\varphi_{n_1+n_2})\end{array}\right).$$
Then ${\bf \hat{m}}$ is an immersion into $\sqrt{\frac{n_1}{n_1+n_2}}\Sigma_1^{n_1}\times\sqrt{\frac{n_2}{n_1+n_2}}\Sigma_2^{n_2}$. The metric of ${\bf \hat{m}}$ is
$$ds^2=\sum_{A,B=1}^{n_1+n_2}\hat{g}_{AB}d\varphi_Ad\varphi_B,$$
where $(\hat{g}_{AB})$ is the block matrix
$$(\hat{g}_{AB})=\left( \begin{array}{c}
\frac{n_1}{n_1+n_2}g_{ab}\\O
\end{array}
\begin{array}{c}
O\\\frac{n_2}{n_1+n_2}g_{a'b'}
\end{array}\right)$$
with
$$a,b=1,\ldots,n_1,\,\,\,\,\, a',b'=n_1+1,\ldots,n_1+n_2,$$
and $$g_{ab}=\frac{\partial{\bf m_1}}{\partial\varphi_a}\cdot\frac{\partial{\bf m_1}}{\partial\varphi_b},\,\,\,\,\,\,\,g_{a'b'}=\frac{\partial{\bf m_2}}{\partial\varphi_{a'}}\cdot\frac{\partial{\bf m_2}}{\partial\varphi_{b'}}.$$
Moreover,
$$(\hat{g}^{AB})=\left( \begin{array}{c}
\frac{n_1+n_2}{n_1}g^{ab}\\O
\end{array}
\begin{array}{c}
O\\\frac{n_1+n_2}{n_2}g^{a'b'}
\end{array}\right)$$
and
$$\hat{g}={\rm det}(\hat{g}_{AB})=\frac{n_1^{n_1}n_2^{n_2}}{(n_1+n_2)^{n_1+n_2}}\,gg',\,\,\,\,\,g={\rm det}(g_{ab}),\,\,\,\,\,g'={\rm det}(g_{a'b'}).$$
Let $\Delta_{n_1},\Delta_{n_2},\Delta_{n_1+n_2}$ denote the Laplacians on $\Sigma_1^{n_1},\Sigma_2^{n_2}$ and $\sqrt{\frac{n_1}{n_1+n_2}}\Sigma_1^{n_1}\times\sqrt{\frac{n_2}{n_1+n_2}}\Sigma_2^{n_2}$, respectively.
Since $\Sigma_1^{n_1},\Sigma^{n_2}_2$ are minimal, we have
$$\Delta_{n_1}{\bf m_1}=\frac{1}{\sqrt{g}}\sum_{a,b}\frac{\partial}{\partial\varphi_a}\left(\sqrt{g}g^{ab}\frac{\partial}{\partial\varphi_b}{\bf m_1}\right)=-n_1{\bf m_1},$$
$$\Delta_{n_2}{\bf m_2}=\frac{1}{\sqrt{g'}}\sum_{a',b'}\frac{\partial}{\partial\varphi_{a'}}\left(\sqrt{g'}g^{a'b'}\frac{\partial}{\partial\varphi_{b'}}{\bf m_2}\right)=-n_2{\bf m_2}.$$
Hence
\begin{eqnarray*}
\Delta_{n_1+n_2}{\bf \hat{m}}&=&\frac{1}{\sqrt{\hat{g}}}\sum_{A,B}\frac{\partial}{\partial\varphi_A}\left(\sqrt{\hat{g}}\hat{g}^{AB}\frac{\partial}{\partial\varphi_B}{\bf \hat{m}}\right)\\
&=&\frac{n_1+n_2}{n_1}\frac{1}{\sqrt{g}}\sum_{a,b}\frac{\partial}{\partial\varphi_a}\left(\begin{array}{c}\sqrt{g}g^{ab}\frac{\partial}{\partial\varphi_b} \sqrt{\frac{n_1}{n_1+n_2}}{\bf m_1}\\O\end{array}\right)\\
&&+\,\,
\frac{n_1+n_2}{n_2}\frac{1}{\sqrt{g'}}\sum_{a',b'}\frac{\partial}{\partial\varphi_{a'}}\left(\begin{array}{c}O\\\sqrt{g'}g^{a'b'}\frac{\partial}{\partial\varphi_{b'}} \sqrt{\frac{n_2}{n_1+n_2}}{\bf m_2}\end{array}\right)\\
&=&-(n_1+n_2)\left(\begin{array}{c}\sqrt{\frac{n_1}{n_1+n_2}}{\bf m_1}\\O\end{array}\right)-(n_1+n_2)\left(\begin{array}{c}O\\\sqrt{\frac{n_2}{n_1+n_2}}{\bf m_2}\end{array}\right)\\
&=&-(n_1+n_2){\bf \hat{m}}.
\end{eqnarray*}
Thus $\sqrt{\frac{n_1}{n_1+n_2}}\Sigma_1^{n_1}\times\sqrt{\frac{n_2}{n_1+n_2}}\Sigma_2^{n_2}$ is minimal.
\end{proof}

{\bf Remark.} Even if $\Sigma_1^{n_1}\subset\mathbb S^{n_1+1}$ and $\Sigma_2^{n_2}\subset\mathbb S^{n_2+1}$ are hypersurfaces, \\ $\sqrt{\frac{n_1}{n_1+n_2}}\Sigma_1^{n_1}\times\sqrt{\frac{n_2}{n_1+n_2}}\Sigma_2^{n_2}$ has codimension 3 in $\mathbb S^{n_1+n_2+3}$. But if $\Sigma_1^{n_1}=\mathbb S^{n_1}$ one can say that $\Sigma_1^{n_1}$ is trivially minimal in $\mathbb S^{n_1}$ and then $\sqrt{\frac{n_1}{n_1+n_2}}\mathbb S^{n_1}\times\sqrt{\frac{n_2}{n_1+n_2}}\Sigma_2^{n_2}$ is minimal with codimension 2 in $\mathbb S^{n_1+n_2+2}$. Furthermore, $\sqrt{\frac{n_1}{n_1+n_2}}\mathbb S^{n_1}\times\sqrt{\frac{n_2}{n_1+n_2}}\mathbb S^{n_2}$ is minimal with codimension 1 in $\mathbb S^{n_1+n_2+1}$.

\section{helicoidal}

Just as the Clifford torus is helicoidal in $\mathbb S^3$, so is the helicoid in $\mathbb R^3$. For a more general setting we introduce the following definition.\\

\noindent{\bf Definition.} Let $M$ be a complete Riemannian manifold and $\Sigma$ an embedded hypersurface of $M$. Assume that $\Sigma$ divides $M$ into two domains $D_1$ and $D_2$. Suppose that at any point $p$ of $\Sigma$ there is an isometry $\varphi$ of $M$ such that
$$\varphi(p)=p, \,\,\,\,\varphi(\Sigma)=\Sigma, \,\,\,\,\varphi(D_1)=D_2,\,\,\,\, \varphi(D_2)=D_1.$$
Then we say that $\Sigma$ is {\it helicoidal} in $M$.

\begin{proposition}\label{clifford}
The generalized Clifford torus $\Sigma^{2p}=\mathbb S^p({1/\sqrt{2}})\times\mathbb S^p({1/\sqrt{2}})$ is helicoidal in $\mathbb S^{2p+1}$.
\end{proposition}
\begin{proof}
Let $\xi$ be the reflection of $\mathbb R^{2p+2}$ defined by
$$\xi(x_1,\ldots,x_{2p+2})=(x_{p+2},x_{p+3},\ldots,x_{2p+2},x_1,x_2,\ldots,x_{p+1}).$$
If $D_1,D_2$ are the domains of $\mathbb S^{2p+1}$ divided by $\Sigma^{2p}$, then
$$\xi(\Sigma^{2p})=\Sigma^{2p},\,\,\,\xi(D_1)=D_2,\,\,\,\xi(D_2)=D_1$$
and $\xi(p)=p$ if and only if
$$p=(x_1,\ldots,x_{p+1},x_1,\ldots,x_{p+1}).$$
For any $q\in\Sigma^{2p}$, there exists an isometry $\eta$ of $\mathbb S^{2p+1}$ mapping $q$ to $p$ such that
$$\eta(\Sigma^{2p})=\Sigma^{2p},\,\,\,\,\eta(D_1)=D_1,\,\,\,\,\eta(D_2)=D_2.$$
Hence
$$\eta^{-1}\circ\xi\circ\eta(q)=q,\,\,\,\,\,\eta^{-1}\circ\xi\circ\eta(\Sigma^{2p})=\Sigma^{2p},$$
and $$\eta^{-1}\circ\xi\circ\eta(D_1)=D_2,\,\,\,\,\,\eta^{-1}\circ\xi\circ\eta(D_2)=D_1.$$
So $\eta^{-1}\circ\xi\circ\eta$ is the desired isometry.
\end{proof}

\begin{theorem}\label{helicoidalthm}
Every helicoidal hypersurface $\Sigma$ of a Riemannian manifold $M^n$ is minimal in $M$ wherever $\Sigma$ is twice differentiable.
\end{theorem}

\begin{proof}
Let $\vec{H}$ be the mean curvature vector of $\Sigma$ at a point $p\in\Sigma$, that is,
$$\vec{H}=\sum_{i=1}^{n-1} (\bar{\nabla}_{e_i}e_i)^\perp,$$
 where $\bar{\nabla}$ is the Riemannian connection on $M$ and $e_1,\ldots,e_{n-1}$ are orthonormal vectors of $\Sigma$ at $p$. Since $\varphi(\Sigma)=\Sigma$ and $p$ is a fixed point of $\varphi$, one sees that $\varphi_*(e_1),\ldots,\varphi_*(e_{n-1})$ are also orthonormal on $\Sigma$ at $p$. Hence
 \begin{equation}\label{H}
 \varphi_*(\vec{H})=\sum_{i=1}^{n-1}(\bar{\nabla}_{\varphi_*(e_i)}\varphi_*(e_i))^\perp=\sum_{i=1}^{n-1} (\bar{\nabla}_{e_i}e_i)^\perp=\vec{H}.
 \end{equation}
 On the other hand, the condition $\varphi(D_1)=D_2$ implies that if $\vec{H}$ points into $D_1$ then $\varphi_*(\vec{H})$ points into $D_2$. Likewise, if $\vec{H}$ points into $D_2$, then $\varphi_*(\vec{H})$ should point into $D_1$. Therefore $\varphi_*(\vec{H})=-\vec{H}$, which together with \eqref{H} implies $\vec{H}=0$ at $p$. As $p$ is arbitrarily chosen, one concludes that $\Sigma$ is minimal.
\end{proof}

Incidentally,
$\mathbb S^1\left({1/\sqrt{2}}\right)\times\mathbb S^1\left({1/\sqrt{2}}\right)$ is congruent in $\mathbb S^3$ to $\mathbb S^3\cap\{(x_1,x_2,x_3,x_4)\in\mathbb R^4:{\rm det}\left(\begin{array}{c}x_1\\x_2
\end{array}\begin{array}{c}
x_3\\x_4\end{array}\right)=0\}$. Also $\mathbb S^2\left({1/\sqrt{2}}\right)\times\,\mathbb S^2\left({1/\sqrt{2}}\right)$ is congruent in $\mathbb S^5$ to $\mathbb S^5\,\cap\,\{(x_1,\ldots,x_6)\in\mathbb R^6:{\rm det}\left(\begin{array}{c}0\\-x_1\\-x_2\\-x_3
\end{array}\begin{array}{c}
x_1\\0\\-x_4\\-x_5\end{array}\begin{array}{c}
x_2\\x_4\\0\\-x_6\end{array}\begin{array}{c}
x_3\\x_5\\x_6\\0\end{array}\right)=0\}.$
When is the zero determinant set minimal? With regard to this question, the following two theorems have been recently proved.

\begin{theorem}\label{T} {\rm (Tkachev, \cite{T})}
$\Sigma=\{(x_{11},x_{12},\ldots,x_{nn})\in\mathbb R^{n^2}:(x_{ij})$ is an $n\times n$ real matrix with zero determinant.$\}$ is a minimal hypercone in $\mathbb R^{n^2}$.
\end{theorem}

\begin{theorem}\label{HLT} {\rm (Hoppe-Linardopoulos-Turgut, \cite{HLT})}
$\Sigma=\{(x_{11},x_{12},\ldots,x_{2n\,2n})\in\mathbb R^{4n^2}:(x_{ij})$ is  a $2n\times2n$ skew-symmetric matrix with zero determinant.$\}$ is congruent to a  minimal hypercone in $\mathbb R^{2n^2-n}$.
\end{theorem}

They obtained these theorems from the harmonicity of $x_{ij}$ on $\Sigma$. Here we will give a new proof by showing that $\Sigma$ is helicoidal.\\

\noindent {\it Proof of Theorem \ref{T}.} Let $M_n$ be the set of all real $n\times n$ matrices. One can identify $M_n$ with $\mathbb R^{n^2}$. Define $\Sigma=\{X\in M_n:{\det}X=0\}$. Then $\Sigma$ is an $(n^2-1)$-dimensional algebraic variety in $\mathbb R^{n^2}$. $\Sigma$ divides $\mathbb R^{n^2}$ into two domains $D_+$ and $D_-$ wih
$$D_+=\{X\in M_n:{\rm det}X>0\},\,\,\,\,D_-=\{X\in M_n:{\rm det}X<0\}.$$
Let's introduce an inner product $\langle\,\,,\,\rangle$ in $M_n$ by
$$\langle X,Y\rangle={\rm tr}(X^TY),\,\,\,\,X,Y\in M_n.$$
Given $A\in O(n)$, define $\varphi_A:M_n\rightarrow M_n$ by $\varphi_A(X)=AX.$ Then $\varphi_A$ is an isometry on $M_n$ because
$$\langle\varphi_A(X),\varphi_A(Y)\rangle=\langle AX,AY\rangle={\rm tr}(X^TA^TAY)={\rm tr}(X^TY)=\langle X,Y\rangle.$$
Clearly
$$\varphi_A(\Sigma)=\Sigma.$$
Moreover, if $A\in SO(n)$, then
$$\varphi_A(D_+)=D_+\,\,{\rm and}\,\,\varphi_A(D_-)=D_-,$$
and if
$A\in O(n)\setminus SO(n),$ then
$$\varphi_A(D_+)=D_-\,\,{\rm and}\,\,\varphi_A(D_-)=D_+.$$
Choose any $X\in\Sigma$. Then the column vectors of $X$ are linearly dependent. Let $P$ be an $(n-1)$-dimensional hyperplane of $\mathbb R^n$ containing all the column vectors of $X$ and let $v\in\mathbb R^n$ be a nonzero normal vector of $P$. Then there exists $A\in O(n)\setminus SO(n)$ such that $P$ is an eigenspace of $A$ with eigenvalue 1 and $v$ an eigenvector of $A$ with eigenvalue $-1$. Hence
$$\varphi_A(X)=X\,\,\,\,{\rm and}\,\,\,\,\varphi_A(D_+)=D_-,\,\,\,\,\varphi_A(D_-)=D_+.$$
Therefore $\Sigma$ is helicoidal in $\mathbb R^{n^2}$ and so by Theorem \ref{helicoidalthm} it is minimal in $\mathbb R^{n^2}$. $\Sigma$ is a cone since ${\rm det}X$ is a homogeneous polynomial.\hspace{5.5cm} $\square$\\

It is known that the determinant of a $2n\times2n$ skew-symmetric matrix $A$ can be written as the square of the Pfaffian of $A$. The Pfaffian $pf(A)$ of $A=(a_{ij})$ is defined as follows. Let $\omega$ be a 2-vector $$\omega=\sum_{i<j}a_{ij}e_i\wedge e_j,$$
 where $\{e_1,\ldots,e_{2n}\}$ is the standard basis of $\mathbb R^{2n}$. Then $pf(A)$ is defined by
 $$\frac{1}{n!}\,\omega^n=pf(A)\,e_1\wedge\cdots\wedge e_{2n}.$$
One computes
 $$pf(A)=\frac{1}{2^nn!}\sum_{\sigma\in S_{2n}}{\rm sgn}(\sigma)\prod_{n=1}^na_{\sigma(2i-1)\sigma(2i)}.$$
Moreover,
\begin{equation}\label{pf}
pf(B^TAB)={\rm det}(B)\,pf(A)
\end{equation}
for any skew-symmetric matrix $A$ and any $2n\times2n$ matrix $B$. \\

{\it Proof of Theorem \ref{HLT}}. Define
$$N=\{X\in M_{2n}:X=\left(\begin{array}{c}0\\-x_1\\-x_2\\\cdot\\\cdot\\-x_{2n-1}
\end{array}\begin{array}{c}
x_1\\0\\-x_{2n}\\\cdot\\\cdot\\-x_{4n-3}\end{array}\begin{array}{c}
x_2\\x_{2n}\\0\\\cdot\\\cdot\\-x_{6n-6}\end{array}\begin{array}{c}
\cdot\\\cdot\\\cdot\\\cdot\\\cdot\\\cdot\end{array}\begin{array}{c}
\cdot\\\cdot\\\cdot\\\cdot\\0\\-x_{2n^2-n}\end{array}\begin{array}{c}
x_{2n-1}\\x_{4n-3}\\x_{6n-6}\\\cdot\\x_{2n^2-n}\\0\end{array}\right)\}$$
and
\begin{center}
$\Sigma=\{X$ is a $2n\times2n$ skew-symmetric marix with ${\rm det}X=0\}.$
\end{center}
Then $\Sigma$ is a hypersurface in the $(2n^2-n)$-dimensional subspace $N$ of $\mathbb R^{4n^2}$. Let
\begin{center}
 $D_+=\{X$ is a $2n\times2n$ skew-symmetric matrix with $pf(X)>0\}$,
 \end{center}
  \begin{center}
  $D_-=\{X$ is a $2n\times2n$ skew-symmetric matrix with $pf(X)<0\}.$
   \end{center}
For any $A\in O(2n)$ define $\psi_A:M_{2n}\rightarrow M_{2n}$ by $$\psi_A(X)=A^TXA.$$
One sees that $\psi_A(X)$ is skew-symmetric if $X$ is. Hence
$$\psi_A:N\rightarrow N\,\,\,\,{\rm and}\,\,\,\,\psi_A(\Sigma)=\Sigma.$$
$\psi_A$ is an isomety since
\begin{eqnarray*}
 \langle\psi_A(X),\psi_A(Y)\rangle&=&\langle A^TXA,A^TYA\rangle={\rm tr}(A^TX^TAA^TYA)\\&=&{\rm tr}(A^TX^TYA)={\rm tr}(AA^TX^TY)=\langle X,Y\rangle.
 \end{eqnarray*}
 Every skew-symmetric matrix can be reduced to a block diagonal form by a special orthogonal matrix. In particular, every $2n\times2n$ skew symmetric matrix $X$ with zero determinant can be transformed by an orthogonal matrix $Q$ to the form
\begin{equation}\label{lambda}
Q^TXQ=\left(\begin{array}{c}0\\-\lambda_1\\0\\0\\0\\0\\0\\0\end{array}
\begin{array}{c}\lambda_1\\0\\\cdot\\\cdot\\\cdot\\\cdot\\\cdot\\\cdot\end{array}\begin{array}{c}
0\\\cdot\\\cdot\\\cdot\\\cdot\\\cdot\\\cdot\\\cdot\end{array}\begin{array}{c}
0\\\cdot\\\cdot\\\cdot\\\cdot\\\cdot\\\cdot\\\cdot\end{array}\begin{array}{c}
0\\\cdot\\\cdot\\\cdot\\0\\-\lambda_k\\\cdot\\\cdot\end{array}\begin{array}{c}
0\\\cdot\\\cdot\\\cdot\\\lambda_k\\0\\\cdot\\\cdot\end{array}\begin{array}{c}
0\\\cdot\\\cdot\\\cdot\\\cdot\\\cdot\\0\\0\end{array}\begin{array}{c}
0\\\cdot\\\cdot\\\cdot\\\cdot\\\cdot\\0\\0\end{array}\right):=\Lambda,
\end{equation}
where $\lambda_1,\ldots,\lambda_k$ are real.

Define a $2n\times2n$ block matrix
$$J=\left(\begin{array}{c}
I_{2n-2}\\{O}\end{array}\begin{array}{c}
{ O}^T\\K\end{array}\right),$$
where $I_{2n-2}$ is the $(2n-2)\times(2n-2)$ identity matrix, $O$ is the $2\times(2n-2)$ zero matrix and $K=\left(\begin{array}{c}
0\\1\end{array}\begin{array}{c}
1\\0\end{array}\right).$
Then for any $X\in\Sigma$ we have an orthogonal matrix $Q$ such that
$$Q^TXQ=\Lambda\,\,\,\,{\rm and}\,\,\,\,J\Lambda J=\Lambda.$$
Hence
$$JQ^TXQJ=Q^TXQ.$$
Therefore
$$(QJQ^T)X(QJQ^T)=X,\,\,\,\,QJQ^T\neq I,\,\,\,\,{\rm det}(QJQ^T)=-1.$$
$QJQ^T$ is orthogonal because
$$(QJQ^T)(QJQ^T)^T=QJQ^TQJQ^T=QJJQ^T=QQ^T=I.$$
Let $B=QJQ^T$. Then by \eqref{pf}
$$pf(\psi_B(Y))=-pf(Y)$$
for any skew-symmetric matrix $Y$
and hence
$$\psi_B(X)=X,\,\,\,\,\psi_B(\Sigma)=\Sigma,\,\,\,\,\psi_B(D_+)=D_-,\,\,\,\,\psi_B(D_-)=D_+.$$
Therefore $\Sigma$ is helicoidal and thus minimal in $N$ everywhere it is twice differentiable.

Let $\mu:N\rightarrow\mathbb R^{2n^2-n}$ be the map defined by
$$\mu(X)=\frac{1}{\sqrt{2}}(x_1,x_2,\ldots,x_{2n^2-n}),$$
where $$X=\left(\begin{array}{c}0\\-x_1\\-x_2\\\cdot\\\cdot\\-x_{2n-1}
\end{array}\begin{array}{c}
x_1\\0\\-x_{2n}\\\cdot\\\cdot\\-x_{4n-3}\end{array}\begin{array}{c}
x_2\\x_{2n}\\0\\\cdot\\\cdot\\-x_{6n-6}\end{array}\begin{array}{c}
\cdot\\\cdot\\\cdot\\\cdot\\\cdot\\\cdot\end{array}\begin{array}{c}
\cdot\\\cdot\\\cdot\\\cdot\\0\\-x_{2n^2-n}\end{array}\begin{array}{c}
x_{2n-1}\\x_{4n-3}\\x_{6n-6}\\\cdot\\x_{2n^2-n}\\0\end{array}\right).$$
Then $\mu$ is an isometry. Therefore $\mu(\Sigma)$ is  a minimal hypercone in $\mathbb R^{2n^2-n}$.\hspace{1.5cm} $\square$\\

\noindent{\bf Questions.}\\
{\bf 1.} A generalized helicoid is defined in \cite{CH} to be the locus of the minimal cone $O\cone(\mathbb S^n(1/\sqrt{2})\times\mathbb S^n(1/\sqrt{2}))$ when the multi-screw motion in $\mathbb R^{2n+3}$ is applied to the cone. That generalized helicoid is minimal. Instead of $\mathbb S^n$, let's consider its  minimal submanifold $M$. Then the cone $O\cone\left(\frac{1}{\sqrt{2}}M\times\frac{1}{\sqrt{2}}M\right)$ is minimal in $\mathbb R^{2n+2}$. If we apply the multi-screw motion in $\mathbb R^{2n+3}$ to the cone, is its locus minimal?\\
{\bf 2.} In the proof of Theorem \ref{T} the hyperplane $P$ is assumed to contain all the column vectors of the matrix $X$. The minimal hypercone $\Sigma$ of the theorem  may have a singularity other than the origin. Is it true that the rank of $X$ is related with the Hausdorff dimension of the singular set of $\Sigma$?\\

\end{document}